\documentclass[12pt]{amsart}

\usepackage{amssymb}
\usepackage{enumerate}

\makeatletter
\@namedef{subjclassname@2010}{%
  \textup{2010} Mathematics Subject Classification}
\makeatother

\newtheorem{thm}{Theorem}[section]

\newtheorem{lem}[thm]{Lemma}
\newtheorem{prop}[thm]{Proposition}

\theoremstyle{definition}

\newtheorem{exa}[thm]{Example}

\newtheorem*{proposition}{Proposition}

\numberwithin{equation}{section}

\begin{document}

\baselineskip=17pt


\title
[On Dirichlet series involving $\zeta(s)$]
{On Dirichlet Series Involving $\zeta(s)$ and Extensions of the Euler--Mascheroni Constant}
\author[T. Noda]{Takumi Noda}
\address{College of Engineering\\ 
Nihon University \\
1 Nakagawara, Tokusada, Tamuramachi, Koriyama, Fukushima 963--8642, Japan}
\email{noda.takumi@nihon-u.ac.jp}

\date{}

\begin{abstract}
In this paper, we introduce a class of Dirichlet series defined in terms of the Riemann zeta-function and establish integral representations for these series.
Motivated by the study of their special values, we define an extension of the Euler--Mascheroni constant and obtain explicit expressions for certain values in terms of the Bendersky constants.
These results provide a unified framework for the evaluation of Dirichlet series involving the Riemann zeta-function at integer arguments, together with the associated number-theoretic constants.
\end{abstract}



\maketitle
\section{Introduction}
As usual, $\zeta(s)$ denotes the Riemann zeta function.
Throughout this paper, we denote by $\mathbb{N}=\{1,2,3,\dots\}$ the set of natural numbers
and by $\mathbb{N}_0=\mathbb{N}\cup\{0\}$ the set of nonnegative integers.
The main object of this paper is the Dirichlet series
\begin{equation}
D(s_1,s_2;\lambda)
:=
\sum_{k=1}^{\infty}
\frac{\zeta(s_1+k)-1}{(k+\lambda)^{s_2}},
\end{equation}
which may be regarded as a generating function associated with the Riemann zeta function.
Here $s_1,s_2\in\mathbb{C}$ with $s_1\notin\mathbb{Z}_{\le0}$, and
$\lambda\in\mathbb{C}$ with $\lambda\notin\mathbb{Z}_{<0}$.
To investigate the special values of $D(s_1,s_2;\lambda)$,
we introduce the function
\begin{equation}
\gamma(\alpha,\beta;\mu)
:=
\lim_{n\to\infty}
\left\{
\sum_{l=1}^{\mu}\frac{1}{l^{\beta}}
\sum_{k=1}^{n}\frac{1}{k^{\alpha+l-\mu}}
-
\sum_{k=2}^{n}
\mathrm{Li}_{\beta}\!\left(\frac1k\right) k^{\mu-\alpha}
\right\},
\end{equation}
which serves as a natural extension of the Euler--Mascheroni constant.
Here $\alpha,\beta\in\mathbb{C}$ with $\Re(\alpha)>0$ and $\mu\in\mathbb{N}_0$,
and $\mathrm{Li}_s(z)=\sum_{l=1}^{\infty}z^l/l^s$ denotes the polylogarithm.
When $\mu=0$, the corresponding sum is omitted.
For $\mu \geqslant 1$,
\begin{equation}
D(\alpha,\beta;\mu)
=
\sum_{l=1}^{\mu}\frac{1}{l^{\beta}}
-
\gamma(\alpha,\beta;\mu).
\end{equation}
This paper establishes fundamental properties of the function $D(s_1,s_2;\lambda)$,
including integral representations (Theorem~2.1 and Lemma~4.1),
and provides closed-form expressions for its special values
(Theorems~2.2--2.4 and~2.6)
in terms of values of the Riemann zeta function and certain classical arithmetic constants.
These results place $D(s_1,s_2;\lambda)$ within a unified framework
that both incorporates earlier studies and reveals structural features
that have not been explicitly formulated before.
\par
Historically, Dirichlet series of the form \((1.1)\) originate in classical results of Goldbach,
communicated in a letter to D.~Bernoulli.
In terms of the Riemann zeta function, Goldbach's theorem may be written as
\begin{equation}
\sum_{k=2}^{\infty}\bigl\{\zeta(k)-1\bigr\}=1,
\end{equation}
which corresponds to the case $D(1,0;1)=1$.
For $s_1=1$, $s_2=1$, and $\lambda=1$, identity \((1.3)\) reduces to Euler's classical formula
\begin{equation}
\sum_{k=2}^{\infty}\frac{\zeta(k)-1}{k}=1-\gamma,
\end{equation}
where $\gamma=\gamma(1,1;1)$ denotes the Euler--Mascheroni constant,
defined by
\begin{equation*}
\gamma:=\lim_{n\to\infty}
\left(\sum_{k=1}^{n}\frac{1}{k}-\log n\right).
\end{equation*}
Following these pioneering works, numerous related formulas were obtained
(see, for example,
\cite{glaisher,glaisher2,ramanujan,ramaswami,apostol}).
A comprehensive historical account of closed-form evaluations of series
involving the Riemann zeta function can be found in
\cite{srivastava} and \cite[Chapter~3]{srivastava-choi}.
Motivated by these developments, it is natural to investigate explicit formulas
for the special values of $D(s_1,s_2;\lambda)$.
In the present paper, our approach relies on formulas due to
Srivastava \cite[(5.3)]{srivastava},
Adamchik--Srivastava \cite[Proposition~4]{adamchik-srivastava},
and Choi--Srivastava \cite[(4.40)]{choi-srivastava},
which will be recalled in Section~2.
%
\section{Main Results}
\paragraph{Notation.}
Throughout this paper, $\Gamma(s)$ and $\zeta(s,\alpha)$ denote the gamma function
and the Hurwitz zeta function, respectively.
The symbol $\int_{+\infty}^{(0+)}$ denotes integration along a Hankel contour,
which starts at positive infinity on the real axis, encircles the origin
counterclockwise with a small radius, and then returns to the starting point.
The generalized harmonic numbers are defined by
$H_m^{(n)}:=\sum_{k=1}^{m}k^{-n}$,
with $H_m:=H_m^{(1)}$ denoting the ordinary harmonic numbers,
and we set $H_0^{(n)}:=0$.
We write $B_j$ for the $j$th Bernoulli number.
The binomial coefficient $\binom{x}{k}$ is defined by
$x(x-1)\cdots(x-k+1)/k!$ for $k\in\mathbb{N}_0$.
\begin{thm}
Let $\lambda\in\mathbb{C}$ with $\lambda\notin\mathbb{Z}_{<0}$.
The Dirichlet series $D(s_1,s_2;\lambda)$ defined in (1.1)
converges absolutely for all $s_1,s_2\in\mathbb{C}$
except when $s_1\in\{0,-1,-2,\dots\}$.
Moreover, it admits the following integral representation:
\begin{equation}
D(s_1,s_2;\lambda)
=
\frac{e^{-\pi i s_2}\Gamma(1-s_2)}
{2\pi i\bigl(e^{2\pi i s_1}-1\bigr)}
\int\limits_{+\infty}^{(0+)}
\int\limits_{+\infty}^{(0+)}
\frac{v^{s_1-1}w^{s_2-1}}
{e^{v+\lambda w}(e^{v}-1)}
\sum_{k=1}^{\infty}
\frac{(v/e^{w})^{k}}{\Gamma(s_1+k)}
\,dv\,dw .
\end{equation}
This representation also defines a holomorphic function on the $(s_1,s_2)$-plane, except for $s_1\in\{0,-1,-2,\dots\}$.
\end{thm}
\begin{thm}
Let $\gamma(\alpha,\beta;\mu)$ be the generalized Euler constant
defined in (1.2).
For $n\in\mathbb{N}$ and $m,\mu\in\mathbb{N}_0$, one has
\begin{equation}
D(n,m;\mu)
=
\begin{cases}
H_{\mu}^{(m)}-\gamma(n,m;\mu),
& \text{if $m>0$,}
\\[4pt]
n-\displaystyle\sum_{k=2}^{n}\zeta(k),
& \text{if $m=0$.}
\\[-4pt]
\end{cases}
\end{equation}
\end{thm}
Since classical times, it has been recognized that
the arithmetic nature of Euler's constant remains unresolved:
its transcendence and the existence of any closed-form representation are unknown,
and even its irrationality has not been established
(for a comprehensive survey, see \cite{ragarias}).
Likewise, a closed-form expression for
$\gamma(n,m;\mu)$
$(n,m\in\mathbb{N},\ \mu\in\mathbb{N}_0)$
also appears to be intractable.
As shown in {(1.5)}, the classical case $D(1,1;1)$
was already evaluated by Euler~\cite{euler}.
Furthermore, for $\mu=2,3$, the corresponding values
$D(1,1;\mu)$
were obtained by Srivastava~\cite[(5.3)]{srivastava}
and Choi--Srivastava~\cite[(4.40)]{choi-srivastava},
in terms of the Stirling and Glaisher--Kinkelin constants, respectively:
\[
\begin{array}{l}
D(1,1;2)= {3}/{2}-{\gamma}/{2}-\log A_0, \\[4pt]
D(1,1;3)= {11}/{6}-{\gamma}/{3}-\log A_0-2\log A_1 .
\end{array}
\]
Here $A_0=\sqrt{2\pi}$ is the Stirling constant, and
$A_1$ is the Glaisher--Kinkelin constant, defined as follows:
\begin{equation*}
\begin{aligned}
\log A_0
&:= \log \sqrt{2\pi}
= \lim_{n\to\infty}
\left\{
\sum_{k=1}^{n}\log k
-\left(n+\frac{1}{2}\right)\log n + n
\right\},
\\
\log A_1
&:= \lim_{n\to\infty}
\left\{
\sum_{k=1}^{n} k\log k
-\left(\frac{n^2}{2}+\frac{n}{2}+\frac{1}{12}\right)\log n
+\frac{n^2}{4}
\right\}.
\end{aligned}
\end{equation*}
The Glaisher--Kinkelin constant and its generalizations originally arose
in the study of the hyperfactorial function and its extensions.
In general, the constants $A_m$ ($m\in\mathbb{N}_{\geqslant 2}$)
are referred to as the Bendersky constants
(cf.~\cite{bendersky};
see also \cite[Section~5]{adamchik}, \cite{copo}, and \cite{perkins-gorder}).
They are defined by the limit
\begin{equation*}
\log A_m
:= \lim_{n\to\infty}
\left\{
\sum_{k=1}^{n} k^{m}\log k - p(n,m)
\right\},
\end{equation*}
where the function $p(n,m)$ is given by
\begin{equation}
\begin{aligned}
p(n,m)
={}&
\left(\frac{n^{m+1}}{m+1}+\frac{n^{m}}{2}\right)\log n
-\frac{n^{m+1}}{(m+1)^2}
\\
&\quad
+ m!\sum_{r=1}^{m}
\frac{B_{r+1}}{(r+1)!}
\frac{n^{m-r}}{(m-r)!}
\left(
\log n
+ (1-\delta_{m,r})
\sum_{i=1}^{r}\frac{1}{m-i+1}
\right),
\end{aligned}
\end{equation}
with $\delta_{m,r}$ denoting the Kronecker delta.
When $m=0$, the summation over $r$ is understood to be omitted.
With these preliminaries in place, we are now ready to state one of our main theorems.
\begin{thm}[A relation between generalized Euler constants and the Bendersky constants]
For $\mu\in\mathbb{N}$, 
\begin{equation}
\gamma(1,1;\mu)
=
\frac{1}{\mu}\gamma
+
\sum_{j=0}^{\mu-2}
\binom{\mu-1}{j}\log A_{j}.
\end{equation}
Consequently,
\begin{equation*}
D(1,1;\mu)
=
H_{\mu}
-\frac{1}{\mu}\gamma
-\sum_{j=0}^{\mu-2}
\binom{\mu-1}{j}\log A_{j}.
\end{equation*}
Here the summation over $j$ is understood to be empty when $\mu=1$.
\end{thm}
We next consider a natural variant of the Stieltjes constants,
denoted by $A_{-1}(a)$,
which will be used to describe the value of $D(1,1;0)$.
\begin{thm}[A relation for the Stieltjes constants]
For $a\in\mathbb{N}_0$, define
\begin{equation*}
\log A_{-1}(a)
:=
\lim_{n\to\infty}
\left\{
\sum_{k=1}^{n}\frac{\log k}{k+a}
-\frac{1}{2}\log^2 n
\right\}.
\end{equation*}
Then
\begin{equation*}
D(1,1;0)
=
\sum_{l=1}^{\infty}(-1)^l\zeta'(l+1)
=
\log A_{-1}(0)-\log A_{-1}(1).
\end{equation*}
Equivalently,
\begin{equation*}
A_{-1}(a)
=
\lim_{n\to\infty}
\prod_{k=1}^{n}
\frac{k^{1/(k+a)}}{n^{\log\sqrt{n}}},
\qquad
e^{D(1,1;0)}
=
\prod_{n=1}^{\infty}
n^{\frac{1}{n}-\frac{1}{n+1}} .
\end{equation*}
\end{thm}
In the above theorem, $\log A_{-1}(0)$ coincides with the first
Stieltjes constant $\gamma_1$
(see, e.g., \cite[1.12\,(17)]{erdelyi}).
Theorems~2.3 and~2.4 show that the case $m=1$ is exceptional.
In contrast, obtaining comparably simple closed-form expressions
for $D(n,m;\mu)$ with $m \geqslant 2$ seems to require
the introduction of a broader class of arithmetic constants.
Nevertheless, one can derive the following identities.
\begin{thm}
Let $n,\mu\in\mathbb{N}$. Then
\begin{equation*}
\begin{aligned}
\sum_{m=0}^{\infty} D(n,m;\mu)
&=
n+H_{\mu-1}
-\gamma(n,1;\mu-1)
-\sum_{k=2}^{n}\zeta(k),
\\
\sum_{m=0}^{\infty}(-1)^m D(n,m;\mu)
&=
n-H_{\mu+1}
+\gamma(n,1;\mu+1)
-\sum_{k=2}^{n}\zeta(k),
\\
\sum_{m=0}^{\infty} D(n,2m+1;\mu)
&=
H_{\mu}
-\frac{1}{2}
\left\{
\gamma(n,1;\mu+1)
+\gamma(n,1;\mu-1)
+\frac{1}{\mu(\mu+1)}
\right\}.
\end{aligned}
\end{equation*}
In particular, when $n=1$, the summation over $k$ is understood
to be omitted.
\end{thm}

We now turn to the case $D(n,-m;\mu)$.
In the fundamental case \mbox{$n=\mu=1$}, Adamchik and Srivastava
obtained the following explicit formula
\cite[Proposition~4]{adamchik-srivastava}
(see also \cite[Proposition~3.8]{srivastava-choi}).
\begin{proposition}
For $m \in \mathbb{N}$, 
\begin{equation*}
D(1,-m;1)=1+ \sum\limits_{k=2}^{m+1}(k-1)! S(m+1,k) \zeta(k).
\end{equation*}
Here $S(m,k)$ denotes the Stirling numbers of the second kind, defined by  
\[
x^m = \sum_{k=0}^{m} S(m,k)\,k!\,\binom{x}{k}.
\]
\end{proposition}
Building on this established result, we obtain the following theorem for $D(n,-m;\mu)$, which reduces to the above proposition
in the case $n=\mu=1$.
\begin{thm}
Let $n,m \in \mathbb{N}$ and $\mu \in \mathbb{N}_0$.
\begin{equation*}
D(n,-m; \mu)
=  
 \sum\limits_{l=0}^{m} 
 \binom{m}{l}
(\mu-n)^{m-l} D(1,-l;1) 
-\sum\limits_{k=2}^{n}\left\{\zeta(k)-1\right\}(k+\mu-n)^m. 
\end{equation*}
\end{thm}
%
We also note that the following index-shifting relations
for $D(n,\pm m;\mu)$ follow directly from Definition~(1.1).
Precisely, for $n,\mu\in\mathbb{N}$ and
$m\in\mathbb{Z}\setminus\{0\}$:
\begin{equation*}
D(n,m;\mu)
=
\begin{cases}
D(n-\mu+1,m;1)
-\displaystyle\sum_{k=2}^{\mu}
\frac{\zeta(k+n-\mu)-1}{k^{m}},
& \text{if $n \geqslant \mu$,}
\\[10pt]
D(1,m;\mu-n+1)
-\displaystyle\sum_{k=2}^{n}
\frac{\zeta(k)-1}{(k+\mu-n)^{m}},
& \text{if $n<\mu$.}
\end{cases}
\end{equation*}
By combining these index-shifting relations with
Theorems~2.3 and~2.4,
we obtain several structural relations among the values
of $D(n,1;\mu)$.
\par
As a concluding result of this section,
we establish a functional relation for
$D(s_1,s_2;\lambda)$ with respect to the variable $s_2$.
In particular, for the values $D(n,-m;0)$
with $n,m\in\mathbb{N}$ and $n>1$,
this relation yields an explicit representation.
\begin{thm}[A functional relation for $D(s_1,s_2;\lambda)$]
Let $s_1,s_2\in\mathbb{C}$ with $s_1\notin\mathbb{Z}_{\leqslant 0}$.
Assume that $\Re(s_1-\lambda)>1$, $\Re(s_2)<0$,
and $-1<\lambda \leqslant 0$.
Then the following functional relation holds:
\begin{equation*}
D(s_1,s_2;\lambda)
=
\sum_{\substack{k\geqslant 2,\, l\in\mathbb{Z}}}
\frac{e^{2\pi i l\lambda}\,\Gamma(1-s_2)}
{k^{\,s_1-\lambda}\,
(2\pi i l-\log k)^{\,1-s_2}}.
\end{equation*}
\end{thm}
%
\section{Further Proposition and Examples}
In addition to the main theorems above,
we include a further proposition and several examples related to
${\gamma}(n, m;\mu)$ and $D(n,m;\mu)$.
We first establish several limit formulas for ${\gamma}(n, m;\mu)$
with respect to its parameters, which will be useful in illustrating
the behavior of these quantities.
\begin{prop}
For $n,m \in \mathbb{N}$ and $\mu \in \mathbb{N}_0$,
\[
\begin{aligned}
\lim_{n \to \infty} \gamma(n,m;\mu)
&= H_{\mu}^{(m)},
\\[6pt]
\lim_{m \to \infty} \gamma(n,m;\mu)
&=
\begin{cases}
1, & \text{if } \mu \geqslant 1,\\
1-\zeta(n+1), & \text{if } \mu = 0,
\end{cases}
\\[10pt]
\lim_{\mu \to \infty} \gamma(n,m;\mu)
&=
\begin{cases}
\zeta(m), & \text{if } m>1,\\[1pt]
\lim\limits_{\mu\to\infty} H_{\mu}, & \text{if } m=1.
\end{cases}
\end{aligned}
\]
\end{prop}
\begin{proof}
Each limit is a direct consequence of the identity (1.3) and Definition~(1.1).
\end{proof}
%
%
\begin{exa}[$D{}(n,m;\mu)$ for $n,m \in \mathbb{N}$, and $\mu \in \mathbb{N}_0$]
\[
\begin{array}{ll}
{D}(1,1;1) 
=1-\gamma, & ({\text{Euler}})  \\[5pt]
{D}(1,1;2) = 
3/2-\gamma/2 - \log A_0, & ({\text{Srivastava}})  \\[5pt]
{D}(1,1;3) = 
{11}/{6}-\gamma/3 - \log A_0 -2\log A_1, & ({\text{Choi--Srivastava}})  \\[5pt]
{D}(1,1;4) 
={25}/{12}-\gamma/4 - \log A_0 -3\log A_1-3\log A_2,&   \\[5pt]
{D}(1,1;0) 
=-\gamma(1,1; 0)
\;= \gamma_1- \log A_{-1}(1),&   \\[5pt]
{D}(2,1;1) 
=1-\zeta(2)+\gamma_1-\log A_{-1}(1). &   \\
\end{array}
\]
\end{exa}

\begin{exa}[$D(n,-m;\mu)$ for $n, \mu \in \mathbb{N}$, and $m \in \mathbb{N}_0$] 
\[ \begin{array}{ll} 
&{D}(1,0;1) = 1, \qquad \qquad \qquad  \qquad \quad \;\; 
({\text{Goldbach}}) \\[5pt] 
&{D}(2,0;1) = 2-\zeta(2), \\[5pt]
& 
\vspace{6pt} 
\begin{cases}
 {D}(1,-1;1) &= 1+\zeta(2), \\[5pt]
 {D}(1,-2;1) &= 1+3\zeta(2)+2\zeta(3), \\[5pt]
  {D}(1,-3;1) &= 1+7\zeta(2)+12\zeta(3)+6\zeta(4), \\[5pt]
  {D}(1,-m;1) &= \cdots , \qquad \qquad  \qquad ({\text{Adamchik--Srivastava}}) \\
   \end{cases} 
   \\
   &{D}(2,-1;1) = \Gamma(2)=1, \\[5pt]
   &{D}(3,-1;1) = \zeta(2)-\zeta(3), \\[5pt]
   &{D}(2,-2;2) = 5-\zeta(2)+2\zeta(3). 
   \end{array} 
   \]
  \end{exa} 
Using \((1.3)\) together with the preceding example, we derive closed-form formulas for
$\gamma(n,-m;1)$.
\begin{exa}
[$\gamma(n,-m;1)$ for $n \in \mathbb{N}$, and $m \in \mathbb{N}_0$]
\[
\begin{array}{l}
\gamma(1,0;1)=0, \\[5pt]
\gamma(1,-1;1)=-\zeta(2), \\[5pt]
\gamma(2,-1;1)=0, \\[5pt]
\gamma(2,-2;1)=-2\zeta(3).
\end{array}
\]
\end{exa}
%
\section{Proofs of Theorems 2.1 and 2.2} 
First, we establish the convergence of
$D(s_1, s_2; \lambda)$.
For $k \geqslant 2$, we have
\[
0<\zeta(k)-1<
\sum\limits_{m=1}^{\infty}\frac{2^m}{2^{mk}}\leqslant\frac{1}{2^{k-2}},
\]
which shows that $\zeta(k)-1$ decays exponentially as $k\rightarrow \infty$.
Hence, the Dirichlet series
\begin{equation*}
D(s_1, s_2; \lambda)= 
\sum\limits_{k=1}^{\infty}\dfrac{\zeta(s_1+k)-1}{(k+\lambda)^{s_2}}
\end{equation*}
converges absolutely for all
$s_1, s_2 \in \mathbb{C}$ and all $\lambda \in \mathbb{C}\setminus\mathbb{Z}_{<0}$
with $s_1 \notin \mathbb{Z}_{\leqslant 0}$.
Next, we derive an integral representation of $D(s_1, s_2; \lambda)$.
To begin with, we prove the following lemma, which will be used in the subsequent analysis.
\begin{lem}
The integral representation
\begin{equation}
D(s_1, s_2;\lambda) = 
\int_{0}^{\infty} 
\dfrac{t^{s_1-1}e^{-t}}{e^{t}-1}
\sum\limits_{k=1}^{\infty}\dfrac{t^k(k+\lambda)^{-s_2}}{\Gamma(s_1+k)}
dt,
\end{equation}
holds for $\Re(s_1)>0$ and $\lambda \notin \mathbb{Z}_{< 0}$.
\end{lem}
\begin{proof}
Using the integral representation of the Hurwitz zeta function
(see, e.g., \cite[p.~266, \S~13.12]{whittaker-watson}), we have
\[
\zeta(s_1+k)-1=\zeta(s_1+k, 2)
=\frac{1}{\Gamma(s_1+k)}\int_{0}^{\infty} 
\frac{t^{s_1+k-1}e^{-t}}{e^t-1}\,dt.
\]
Substituting this into (1.1) and 
interchanging the order of summation and integration, 
we obtain a part of the integrand that will be estimated below:
\[
\begin{array}{rll}
\displaystyle
\sum_{k=1}^{\infty}
\left| \frac{t^k (k+\lambda)^{-s_2}}{\Gamma(s_1+k)} \right|
&< M_1\, t,
& \text{if } t \in [0,1], \\[4pt]
\displaystyle
\sum_{k=1}^{\infty}
\left| \frac{t^k (k+\lambda)^{-s_2}}{\Gamma(s_1+k)} \right|
&< M_2\, e^{(1+\epsilon)t},
& \text{if } t \in [1,\infty),
\end{array}
\]
where $M_1$ and $M_2$ are positive constants depending only on $s_2$ and $\lambda$,
and $\epsilon$ is a small positive constant.
Since the integrand is dominated by an absolutely integrable function on $[0,\infty)$,
the interchange of summation and integration is justified by Fubini's theorem,
and hence we obtain the expression (4.1).
This completes the proof of the lemma.
\end{proof}
%
\begin{proof}[Proof of Theorem 2.1]
We start with the integral representation of $\Gamma(z)$ along the Hankel contour (see, e.g., \cite[1.6 (2)]{erdelyi}):
\begin{equation}
\frac{1}{\Gamma(z)}
=\frac{e^{\pi i}}{2\pi i}\int_{\infty}^{(+0)} 
{(-t)^{-z}e^{-t}}\,dt,
\end{equation}
where $\arg (-t)=-\pi +\arg t$.  
After changing the variable $t$ to $(k+\lambda)w$ and shifting the path of integration back to the Hankel contour, we obtain 
\begin{equation}
\frac{1}{(k+\lambda)^{s_2}}
=\frac{\Gamma(1-s_2)}{2\pi i\, e^{\pi i s_2}}\int_{\infty}^{(+0)} 
{w^{s_2-1}e^{-(k+\lambda)w}}\,dw.
\end{equation}
In a similar manner, (4.2) yields the following integral representation of the Hurwitz zeta function
(see, e.g., \cite[p.~266, \S~13.13]{whittaker-watson}):
\begin{equation}
\zeta(s_1+k)-1=\zeta(s_1+k, 2)
=\frac{\Gamma(1-s_1-k)}{2\pi i\, e^{\pi i (s_1+k)}}\int_{\infty}^{(+0)} 
\frac{v^{s_1+k-1}e^{-v}}{e^v-1}\,dv.
\end{equation}
Substituting (4.4) into (1.1) and interchanging the order of summation and integration,
which is justified by a similar argument to that in the proof of Lemma~4.1,
then substituting (4.3) and applying the reflection formula
$\Gamma(s)\Gamma(1-s)=\pi / \sin(\pi s)$,
we obtain the desired integral representation~(2.1).
\par
In (2.1), possible poles arise when $s_1 \in \mathbb{Z}$ and $s_2 \in \mathbb{N}$. 
However, when $s_1, s_2 \in \mathbb{N}$, the $v$- and $w$-integrals vanish respectively, and the apparent singularities are canceled. 
This observation agrees with the result of Lemma~4.1. 
Hence, it defines a holomorphic function on the $(s_1, s_2)$-plane except for $s_1 \in \{0, -1, -2, \dots\}$.
\end{proof}
%
\begin{proof}[Proof of Theorem 2.2]
Let $n \in \mathbb{N}$ and $m, \mu \in \mathbb{N}_0$. 
First, suppose that $m>0$. Then we have
\begin{equation*}
\begin{array}{ll}
 D(n, m; \mu) 
 & 
=
\sum\limits_{k=2}^{\infty}
\sum\limits_{l=1}^{\infty}  
 \dfrac{1}{(l+\mu)^m} {\left( \dfrac{1}{k}\right)^{n+l}} 
 = 
\sum\limits_{k=2}^{\infty}
  {\left( \dfrac{1}{k}\right)^{n-\mu }} 
\!\!
\sum\limits_{l=\mu+1}^{\infty}  
 \dfrac{1}{l^m} {\left( \dfrac{1}{k}\right)^{l}} 
 \\[15pt] 
 & 
 =
\sum\limits_{k=2}^{\infty}
  {\left( \dfrac{1}{k}\right)^{n-\mu }} 
 \Big\{
 \mathrm{Li}_m\left( \dfrac{1}{k}\right)
 -\sum\limits_{l=1}^{\mu}  
 \dfrac{1}{l^m} {\left( \dfrac{1}{k}\right)^{l}} 
 \Big\}.
\end{array}
\end{equation*}
 This is equivalent to
 \begin{equation*}
\sum\limits_{l=1}^{\mu} \dfrac{1}{l^m}
 +
 \lim\limits_{N \to \infty}
 \left\{ 
  \sum\limits_{k=2}^{N}
 \mathrm{Li}_m\left( \dfrac{1}{k}\right)
  {\left( \dfrac{1}{k}\right)^{n-\mu }} 
 -\sum\limits_{l=1}^{\mu}  
 \dfrac{1}{l^m}
  \sum\limits_{k=1}^{N}
  {\left( \dfrac{1}{k}\right)^{n+l-\mu}} 
 \right\}.
 \end{equation*}
Here, in the case $\mu = 0$, the summation over $l$ is understood to be absent.
For the case $m=0$, we see that $D(n, m; \mu) $ is equal to
 \begin{equation*}
 \sum\limits_{l=1}^{\infty}\left\{\zeta(n+l)-1\right\}
 = \sum\limits_{l=2}^{\infty}\left\{\zeta(l)-1\right\} 
   -\sum\limits_{l=2}^{n}  \left\{  \zeta(l) -1\right\} 
   =n-\sum\limits_{l=2}^{n} \zeta(l).
\end{equation*}
In the last equality, we have used formula (1.4) due to Goldbach.
This completes the proof of Theorem~2.2.
\end{proof}
\section{Proof of Theorem 2.3} 
In this section, we prove Theorem~2.3.
The argument is technically involved and unfolds through several steps.
Our approach repeatedly uses changes of summation indices,
rearrangements of double sums, and the recurrence formula for the Bernoulli numbers.
Among these ingredients, a key step is the use of a recurrence relation
for the Bernoulli numbers involving two different types of sums.
Since its proof is technical and relies on Agoh's theorem
\cite[Theorem~3.1]{agoh}, we establish this relation first.
\begin{prop}
Let $m \geqslant 2$. 
Then we have the following identity:
\begin{equation}
\sum\limits_{j=0}^{m-1}
\left[
 \dfrac{(-1)^{j}}{m-j}
+
\binom{m}{j}H_j
\right]
B_j=0.
\end{equation}
More generally, for integers $m > k \geqslant 1$,
\begin{equation}
\begin{aligned}
&
\sum\limits_{j=0}^{m-k}
\left[
\dfrac{(-1)^{j} }{(m-j-k+1)(j+k)} \dbinom{j+k}{j}
 \right.
 \\
 & \qquad \qquad 
+
  \left.
  \dfrac{1}{m+1}\dbinom{m+1}{k}
\dbinom{m-k+1}{j}H_{j+k-1}
 \right]B_j
 =0.
\end{aligned}
\end{equation}
\end{prop}

\begin{proof} 
We prove the general identity~(5.2), from which~(5.1) follows as the special case $k=1$.
The identity~(5.1) itself is essentially equivalent to \cite[Corollary~3.3(iii)]{agoh}; see also \cite[Theorem~7 (2.17)]{qin-lu}, upon using the facts that $B_1=-\tfrac12$ and $B_{2n+1}=0$ for $n\in\mathbb N$.

Heuristically, \((5.2)\) may be viewed as arising from the $(k-1)\mbox{-th}$ derivative of the polynomial identity associated  with~(5.1).
The proof of~(5.2) follows the arguments of \cite[Lemma~2.2, 2.3(iii)]{agoh}, applies Agoh's theorem~\cite[Theorem~3.1]{agoh},
and thereby yields results analogous to
\cite[Corollary~3.2(iii), 3.3(iii)]{agoh}.

We start from the well-known identity
\[
\sum_{j=1}^{n}\frac{(-1)^j}{j}\binom{n}{j}=-H_n,
\]
from which it follows that
\begin{equation*}
\sum_{j=1}^{n}\frac{(-1)^j}{j}\binom{m}{m-j}\binom{m-j}{m-n}
=-\binom{m}{n}H_n.
\end{equation*}
Here, the left-hand side can be interpreted as the coefficient of $x^n$ in
\[
\sum_{r=0}^{m-1}\binom{m}{r}\frac{(x+1)^r(-x)^{m-r}}{m-r}.
\]
We are therefore led to
\begin{equation*}
\begin{aligned}
\frac{1}{k}\binom{j+k-1}{j}
&\sum_{l=1}^{j+k-1}
\frac{(-1)^l}{l}\binom{m}{l}\binom{m-l}{m-j-k+1} \\
&=
-\frac{1}{k}\binom{j+k-1}{j}\binom{m}{j+k-1}H_{j+k-1},
\end{aligned}
\end{equation*}
where the left-hand side can be identified with the coefficient of $x^{j}$ in
\begin{equation*}
\frac{1}{k!}\,
\frac{d^{\,k-1}}{dx^{k-1}}
\left(
\sum_{r=0}^{m-1}\binom{m}{r}\frac{(x+1)^r(-x)^{m-r}}{m-r}
\right).
\end{equation*}
After shifting $x$ to $x-1$ and replacing $r$ by $m-r$, we obtain
\begin{equation*}
\begin{aligned}
&\frac{1}{k!}\,
\frac{d^{\,k-1}}{dx^{k-1}}
\left(
\sum_{r=1}^{m}\binom{m}{r}\frac{x^{m-r}(1-x)^{r}}{r}
\right) \\
&= 
-
\sum\limits_{j=0}^{m-k+1} \frac{1}{k}\binom{j+k-1}{j}\binom{m}{j+k-1}H_{j+k-1}
(x-1)^{j}.
\end{aligned}
\end{equation*}
Substituting the identity
\begin{equation*}
\sum_{r=1}^{m}\binom{m}{r}\frac{x^{m-r}(1-x)^{r}}{r}
=
\sum_{r=0}^{m-1} \frac{x^r}{m-r} - H_{m}x^{m},
\end{equation*}
which is given in \cite[Proof of Lemma~2.3(iii)]{agoh},
into the above equality, and replacing $r$ by $j+k-1$,
together with a rearrangement of the binomial coefficients
$\binom{m}{j+k-1}\binom{j+k-1}{j}$,
we obtain
\begin{equation}
\begin{aligned}
& \sum\limits_{j=0}^{m-k}\frac{1}{(m-j-k+1)(j+k)} \binom{j+k}{j}x^j
\\
& \qquad + \sum\limits_{j=0}^{m-k+1}\dfrac{1}{m+1}\binom{m+1}{k}
\binom{m-k+1}{j}H_{j+k-1}(x-1)^j,
\\
& \qquad = \frac{1}{m+1}\binom{m+1}{k}H_{m}x^{m-k+1}.
\end{aligned}
\end{equation}
We next introduce
\begin{equation*}
\begin{aligned}
& f(x)=\sum\limits_{j=0}^{m-k}\dfrac{1}{(m-j-k+1)(j+k)} \dbinom{j+k}{j}x^j,
\\
& g(x)=\sum\limits_{j=0}^{m-k}\dfrac{1}{m+1}\dbinom{m+1}{k}
\dbinom{m-k+1}{j}H_{j+k-1}x^j,
\\
\end{aligned}
\end{equation*}
and define $h_1(x):=f(x+1)+g(x)$. By (5.3), we obtain
\begin{equation*}
\begin{aligned}
h_1(x)=\dfrac{1}{m+1}\dbinom{m+1}{k}H_{m}
\left[ (x+1)^{m-k+1}-x^{m-k+1} \right].
\end{aligned}
\end{equation*}
It is readily verified that
\begin{equation*}
\begin{aligned}
\dfrac{d}{dx}f(x)=\sum\limits_{j=1}^{m-k}\dfrac{1}{m-j-k+1} \dbinom{j+k-1}{j-1}x^{j-1}.
\end{aligned}
\end{equation*}
Under these preparations, we now apply Agoh's theorem \cite[Theorem 3.1]{agoh}, which asserts that
\begin{equation*}
f(B(x))+g(B(x))=h_1(B(x))-\dfrac{d}{dx}f(x)
\quad
\text{with} \quad 
B(x)^k=B_k(x),
\end{equation*}
where $B_k(x)$ is the $k$th Bernoulli polynomial.
It follows that
\begin{equation}
\begin{aligned}
& \sum\limits_{j=0}^{m-k}\dfrac{1}{(m-j-k+1)(j+k)} \dbinom{j+k}{j}B_j(x)
\\
& \qquad \qquad +
\sum\limits_{j=0}^{m-k}\dfrac{1}{m+1}\dbinom{m+1}{k}
\dbinom{m-k+1}{j}H_{j+k-1}B_j(x)
\\
& = \dfrac{1}{m+1}\dbinom{m+1}{k}H_{m} \,(m-k+1)x^{m-k}\\
& \qquad \qquad
 - \sum\limits_{j=1}^{m-k}\dfrac{1}{m-j-k+1} \dbinom{j+k-1}{j-1}x^{j-1}.
\end{aligned}
\end{equation}
Here we have used the standard identity for Bernoulli polynomials
($n\in\mathbb N$):
\[
\left( B(x)+1\right)^n-B(x)^n
=
\sum\limits_{j=0}^{n-1}\dbinom{n}{j}B_j(x)
=nx^{n-1}.
\]
Noting that $m>k$, evaluating (5.4) at $x=0$ yields
\begin{equation*}
\begin{aligned}
&
\sum\limits_{j=0}^{m-k}
\left[
\dfrac{1}{(m-j-k+1)(j+k)} \dbinom{j+k}{j}
 \right.
 \\
 & \qquad \qquad 
+
  \left.
  \dfrac{1}{m+1}\dbinom{m+1}{k}
\dbinom{m-k+1}{j}H_{j+k-1}
 \right]B_j
 =-\dfrac{1}{m-k}.
\end{aligned}
\end{equation*}
Using $B_1=-\tfrac12$ and $B_{2n+1}=0$ for $n\in\mathbb N$,
we complete the proof of (5.2).
\end{proof}
We now prove the following lemma for the Bendersky constant 
$A_m$,  
which underlies the proof of Theorem~2.3.
\begin{lem}
Let $n, m \in \mathbb{N}$.
For the function $P(n,m)$ defined in (2.3), 
\begin{equation*}
\sum\limits_{j=0}^{m-1}\binom{m}{j}P(n, j)
=
(n+1)^m\log n
-\dfrac{\log n}{m+1}
-\sum\limits_{l=1}^{m}\dfrac{1}{l} \sum\limits_{k=1}^{n}{k^{m-l}}.
\end{equation*}
\end{lem}
\begin{proof}
By the definition of $P(n,m)$, we have
\begin{equation*}
\sum\limits_{j=0}^{m-1}\dbinom{m}{j}P(n, j) 
= S_1(n,m)-S_2(n,m)+S_3(n,m)+S_4(n,m),
\\[-10pt]
\end{equation*}
\begin{equation*}
\begin{array}{rl}
\text{where} \quad
 S_1(n,m) &=  
\sum\limits_{j=0}^{m-1}\dbinom{m}{j} 
\left(\dfrac{n^{j+1}}{j+1} + \dfrac{n^j}{2}\right) \log n,
\\[12pt]
S_2(n,m) &= \sum\limits_{j=0}^{m-1}\dbinom{m}{j} \dfrac{n^{j+1}}{(j+1)^2},
\\[12pt]
S_3(n,m) &= \sum\limits_{j=1}^{m-1}
\dfrac{m!}{(m-j)!}\sum\limits_{r=1}^{j}
\dfrac{B_{r+1}}{(r+1)!}
\dfrac{n^{j-r}}{(j-r)!}
\log n,
\\[12pt]
S_4(n,m) &= \sum\limits_{j=1}^{m-1}
\dfrac{m!}{(m-j)!}\sum\limits_{r=1}^{j}
\dfrac{B_{r+1}}{(r+1)!}
\dfrac{n^{j-r}}{(j-r)!}
\sum\limits_{i=1}^{r}\dfrac{1-\delta_{j,r}}{j-i+1}.
\end{array}
\end{equation*}
Here, when $m = 1$, the summations $S_3$ and $S_4$ are understood to be zero.

First, we evaluate $S_1$ and $S_3$ separately. 
By applying the binomial theorem, we obtain
\begin{equation*}
 S_1(n,m) 
 = 
(1+n)^{m}\log n+
\sum\limits_{j=1}^{m-1}\dbinom{m}{j} 
\left(\dfrac{1}{m-j+1} - \dfrac{1}{2}\right) n^j \log n -\dfrac{\log n}{2}.
\end{equation*}
%
In the simplification of $S_3(n,m)$, we proceed in three steps.
In the first step, we change the index by setting $k=j-r$, 
which yields $k=0,1,\dots, m-1-r$. 
In the second step, we interchange the order of summations by using the 
elementary identity 
 $  \binom{m}{r}\binom{m-r}{k} = \binom{m}{k}\binom{m-k}{r}.  $
In the third step, we apply the identity
\(\sum_{j=0}^{k}\binom{k+1}{j}B_j=0\).
After these transformations, the sum can be rewritten as follows:
\begin{equation*}
\begin{array}{ll}
 S_3(n,m) 
 &  = 
 \sum\limits_{r=1}^{m-1}
\dfrac{B_{r+1}}{r+1} \sum\limits_{k=0}^{m-1-r}
\dbinom{m}{r} \dbinom{m-r}{k}
n^{k}\log n
\\[15pt]
 &  = 
 \sum\limits_{k=0}^{m-2}
\dbinom{m}{k}  \sum\limits_{r=1}^{m-1-k}
\dfrac{B_{r+1}}{r+1} 
\dbinom{m-k}{r}
n^{k}\log n
\\[15pt]
 &  = 
 \sum\limits_{k=0}^{m-2}
\dbinom{m}{k}  
\dfrac{n^{k}\log n}{m-k+1}
\sum\limits_{j=2}^{m-k}
\dbinom{m-k+1}{j}{B_{j}}
\\[15pt]
 &  = 
 \sum\limits_{k=0}^{m-2}
\dbinom{m}{k}  
\dfrac{n^{k}\log n}{m-k+1}
\left\{
-B_0-(m-k+1)B_1
\right\}
\\[15pt]
 &  = 
 \sum\limits_{k=0}^{m-2}
\dbinom{m}{k}  
\left(-\dfrac{1}{m-k+1}
+\dfrac12
\right)
n^{k}\log n.
\\[15pt]
\end{array}
\end{equation*}
%
We thus arrive at the following first key identity:
\begin{equation}
 S_1(n,m) + S_3(n,m) 
 = 
(1+n)^{m}\log n-\dfrac{\log n}{m+1}.
\end{equation}
Next, we establish the second key identity:
\begin{equation}
 S_2(n,m) - S_4(n,m) 
 = 
\sum\limits_{l=1}^{m}\dfrac{1}{l} \sum\limits_{k=1}^{n}{k^{m-l}}.
\end{equation}
For simplicity, we denote the right-hand side of the above equality by $\mathcal{H}_{n,m}$.
Before proving the second identity (5.6), we need a few preliminary steps. 
To begin with, note that the following equality is obtained by first applying Faulhaber's formula and then changing the summation variables.
\begin{equation*}\begin{array}{ll}
\mathcal{H}_{n,m}
 & =
\sum\limits_{r=0}^{m-1}\dfrac{1}{m-r} 
\sum\limits_{k=1}^{n}{k^{r}}
\\[15pt]
 & =
\sum\limits_{r=0}^{m-1}\dfrac{1}{(m-r)(r+1)} \sum\limits_{i=0}^{r}
\dbinom{r+1}{i}(-1)^{\delta_{i,1}}B_i \,n^{r+1-i}
\\[15pt]
 & = 
\sum\limits_{k=1}^{m}\sum\limits_{j=0}^{m-k}
\dfrac{(-1)^{\delta_{j,1}}B_j }{(m-j-k+1)(j+k)} 
\dbinom{j+k}{j}n^{k}.
\end{array}
\end{equation*}
%
For $S_2(n,m)$, we can write
\begin{equation}
 S_2(n,m)   = \dfrac{1}{m+1}
\sum\limits_{k=1}^{m}\dbinom{m+1}{k} \dfrac{n^{k}}{k}.
\end{equation}
It is easily verified that the coefficients of $n^{m}$ and $n^{m-1}$ in 
$\mathcal{H}_{n,m}$ are 
${1}/{m}$ and ${m}/{2(m-1)}$, respectively. 
These coincide with the coefficients of $n^{m}$ and $n^{m-1}$ in $S_2(n,m)$. 
Since $S_4(n,m)$ is a polynomial in $n$ of degree at most $m-2$, it suffices to show that the coefficients of $n^{k}$ for $k=1,\dots,m-2$ 
in $\mathcal{H}_{n,m} + S_4(n,m)$ agree with the corresponding coefficients of $n^{k}$ in $S_2(n,m)$.
For $S_4(n,m)$, by applying the same approach as in the simplification of $S_3(n,m)$,  
we rewrite this summation as follows:
\begin{equation*}
\begin{array}{ll}
 S_4(n,m) 
 &  = 
 \sum\limits_{r=1}^{m-1}
\dfrac{B_{r+1}}{r+1} \sum\limits_{k=0}^{m-1-r}
\dbinom{m}{r} \dbinom{m-r}{k}
n^{k}\,\sum\limits_{i=1}^{r}\dfrac{1-\delta_{k,0}}{r+k-i+1}
\\[15pt]
 &  = 
 \sum\limits_{k=1}^{m-2}
\dbinom{m}{k}  \sum\limits_{r=1}^{m-1-k}
\dfrac{B_{r+1}}{r+1} 
\dbinom{m-k}{r}
n^{k}\,\sum\limits_{i=1}^{r}\dfrac{1}{r+k-i+1}
\\[15pt]
 &  = 
 \sum\limits_{k=1}^{m-2}
\dbinom{m}{k}  
\dfrac{n^{k}}{m-k+1}
\sum\limits_{j=2}^{m-k}
\dbinom{m-k+1}{j}{B_{j}}
\sum\limits_{i=1}^{j-1}\dfrac{1}{j+k-i}.
\end{array}
\end{equation*}
Hence, the coefficients of $n^{k}$ in $\mathcal{H}_{n,m} + S_4(n,m)$ are given by
\begin{equation*}
\begin{aligned}
&
\sum\limits_{j=0}^{m-k}
\frac{(-1)^{\delta_{j,1}}B_j }{(m-j-k+1)(j+k)} 
\binom{j+k}{j}
\\
 & \qquad +
\dbinom{m}{k}  
\frac{1}{m-k+1}
\sum\limits_{j=2}^{m-k}
\binom{m-k+1}{j}{B_{j}}\left(H_{j+k-1}-H_{k}\right).
\end{aligned}
\end{equation*}
By the recurrence formula $ \sum_{j=0}^{m-k}\binom{m-k+1}{j}{B_{j}} =0$, and noting that $H_{k}-H_{k-1}=1/k$, the expression above can be rewritten as
\begin{equation*}
\begin{aligned}
&
\sum\limits_{j=0}^{m-k}
\left[
\frac{(-1)^{\delta_{j,1} } }{(m-j-k+1)(j+k)} \binom{j+k}{j}
 \right.
 \\
 & \quad 
+
  \left.
  \frac{1}{m+1}\binom{m+1}{k}
\binom{m-k+1}{j}H_{j+k-1}
 \right]B_j
 +
\frac{1}{k(m+1)}
\binom{m+1}{k}.
\end{aligned}
\end{equation*}
Upon substituting Proposition~5.1 into the above expression, we see that it
coincides with the coefficient of $n^k$ in $S_2(n,m)$ given in~(5.7).
Hence,~(5.6) follows.
Together with~(5.5), this completes the proof of Lemma~5.2.
\end{proof}
%
%
\begin{proof}[Proof of Theorem 2.3]
Applying the binomial expansion to the definition of $\gamma(\alpha, \beta; \mu)$ in (1.2), we obtain
\begin{equation*}
\begin{array}{ll}
& \gamma (1, 1; \mu)
=  \lim\limits_{n\to \infty}
\left\{
\sum\limits_{l=1}^{\mu}\dfrac{1}{l}
\sum\limits_{k=1}^{n} \left( \dfrac{1}{k} \right)^{1+l-\mu}
\!\!\! +
\sum\limits_{k=2}^{n}k^{\mu-1}\log
\Big( \dfrac{k-1}{k} \Big) \right\}
\\[15pt]
= &
\lim\limits_{n\to\infty}
\left\{
\sum\limits_{l=1}^{\mu}\dfrac{1}{l}
  \sum\limits_{k=1}^{n} \left( \dfrac{1}{k} \right)^{1+l-\mu}
\!\!\! +
\sum\limits_{k=2}^{n-1}
  \bigl((k+1)^{\mu-1}-k^{\mu-1}\bigr)\log k
-
n^{\mu-1}\log n
\right\}
\\[15pt]
= & \lim\limits_{n\to \infty}
\left\{
\sum\limits_{l=1}^{\mu}\dfrac{1}{l}
\sum\limits_{k=1}^{n}{k^{\mu-l-1}} 
+
\sum\limits_{k=2}^{n}
\sum\limits_{j=0}^{\mu -2} \dbinom{\mu -1}{j} k^j 
\log k
- (n+1)^{\mu-1} \log n
\right\}.
\end{array}
\end{equation*}
On the other hand, 
Lemma~5.2 shows that the right-hand side of (2.4) becomes
\begin{equation*}
\begin{array}{ll}
& \dfrac{1}{\mu}\gamma+\sum\limits_{j=0}^{\mu-2}
\dbinom{\mu-1}{j}\log A_{j}
\\[15pt]
= &
\dfrac{1}{\mu}\gamma+
\lim\limits_{n\to \infty}
\left\{
\sum\limits_{k=1}^{n}
\sum\limits_{j=0}^{\mu -2} \dbinom{\mu -1}{j}
k^j\log k
-
\sum\limits_{j=0}^{\mu -2} \dbinom{\mu -1}{j}
p(n, j)
\right\}
\\[15pt]
= &
\lim\limits_{n\to \infty}
\left\{
\sum\limits_{k=1}^{n}
\sum\limits_{j=0}^{\mu -2} \dbinom{\mu -1}{j}
k^j\log k
-
(n+1)^{\mu -1}\log n
+\sum\limits_{l=1}^{\mu }\dfrac{1}{l} \sum\limits_{k=1}^{n}{k^{\mu-l-1}}
\right\}.
\end{array}
\end{equation*}
Therefore, the identity (2.4) holds, and this completes the proof of Theorem~2.3.
\end{proof}

\section{Proofs of Theorems~2.4 -- 2.6}
%
%
\begin{proof}[Proof of Theorem 2.4]
Following the same argument as in the proof of Theorem~2.2, we obtain
\begin{equation*}
\begin{aligned}
 D(1, 1; 0) 
 & =
 \textstyle\sum\limits_{k=2}^{\infty}
 k^{-1}
 \mathrm{Li}_1\left(k^{-1}\right)
 \;\; = \textstyle\sum\limits_{k=2}^{\infty}
k^{-1} (\log k - \log(k-1))
\\
& = \textstyle\sum\limits_{k=1}^{\infty}\dfrac{\log k}{k(k+1)}
\;
= \lim\limits_{n\to \infty}
\left( \textstyle\sum\limits_{k=1}^{n}\dfrac{\log k}{k}- \textstyle\sum\limits_{k=1}^{n}\dfrac{\log k}{k+1}
\right)
\\[4pt]
& = \log A_{-1}(0)-\log A_{-1}(1)
\;\;  =  \gamma_1-\log A_{-1}(1).
 \end{aligned}
\end{equation*}
Furthermore, we observe that
\[
\sum_{l=1}^{\infty} (-1)^l \zeta'(l+1)
=
\sum_{l=1}^{\infty} \sum_{k=1}^{\infty}
\frac{(-1)^{\,l+1}}{k^{\,l+1}} \log k
=
\sum_{k=1}^{\infty} \frac{\log k}{k(k+1)}.
\]
We note that $\log A_{-1}(a)$ 
is well defined for $a \in \mathbb{N}$, as well as for the first Stieltjes constant
$\gamma_1 = \log A_{-1}(0)$ by applying the Euler--Maclaurin summation formula;
this can also be observed from the convergence of the series
$\sum_{k=1}^{\infty} {\log k}/{k(k+a)}$.
This concludes the proof of Theorem~2.4.
\end{proof}
\begin{proof}[Proof of Theorem 2.5]
For $n, \mu \in \mathbb{N}$, we have
\begin{equation*}
\begin{aligned}
& \textstyle\sum\limits_{m=0}^{\infty}D(n,m;\mu)
\;  =  \;
\textstyle\sum\limits_{k=1}^{\infty}\{\zeta(n+k)-1\}
\textstyle\sum\limits_{m=0}^{\infty}
\dfrac{1}{(k+\mu)^{m}}
 \\
=  & 
\textstyle\sum\limits_{k=1}^{\infty}\{\zeta(n+k)-1\}
\Big( 1+\dfrac{1}{k+\mu-1} \Big)
 =  
D(n,0;\mu)+D(n,1; \mu -1)
\\
= & 
n-\textstyle\sum\limits_{k=2}^{n}\zeta(k)+H_{\mu-1}-\gamma(n,1;\mu -1).
\end{aligned}
\end{equation*}
Similarly, we have
\begin{equation*}
\begin{aligned}
& \textstyle\sum\limits_{m=0}^{\infty}(-1)^m D(n,m;\mu)
 \;  =  \; 
\textstyle\sum\limits_{k=1}^{\infty}\{\zeta(n+k)-1\}
\textstyle\sum\limits_{m=0}^{\infty}
\Big( \dfrac{-1}{k+\mu}\Big)^{m}
 \\
 =  & 
\textstyle\sum\limits_{k=1}^{\infty}\{\zeta(n+k)-1\}
\Big( 1-\dfrac{1}{k+\mu+1} \Big)
 = 
D(n,0;\mu)-D(n,1; \mu +1)
\\
 = &
n-\textstyle\sum\limits_{k=2}^{n}\zeta(k)-H_{\mu+1}+\gamma(n,1;\mu +1).
\end{aligned}
\end{equation*}
Here we have applied Theorem~2.2.
By combining the above identities, we obtain
\begin{equation*}
\begin{aligned}
& \textstyle\sum\limits_{m=0}^{\infty}D(n,2m+1;\mu)
= 
\lim\limits_{N\to \infty}\sum\limits_{m=0}^{N}
 \dfrac{1-(-1)^m}{2}\, D(n,m;\mu)
 \\
= &  H_{\mu}-\dfrac12\Big\{ \gamma(n,1;\mu +1)+\gamma(n,1;\mu -1) +\dfrac{1}{\mu}-\dfrac{1}{\mu+1}\Big\}.
\end{aligned}
\end{equation*}
This completes the proof of Theorem~2.5.
\end{proof}
%
\begin{proof}[Proof of Theorem 2.6]
Applying the binomial expansion, we have
\begin{equation*}
\begin{aligned}
& D(n,-m; \mu)
\;  =  \; 
\textstyle\sum\limits_{k=1}^{\infty}\{\zeta(n+k)-1\}
(k+\mu)^{m}
 \\ 
=  & 
\textstyle\sum\limits_{k=1}^{\infty}\{\zeta(n+k)-1\}
\textstyle\sum\limits_{l=0}^{m}\dbinom{m}{l} (\mu-n)^{m-l}(k+n)^{l}
\\ 
= & 
 \textstyle\sum\limits_{l=0}^{m} \dbinom{m}{l} (\mu-n)^{m-l} D(n, -l; n)
 \\
= & 
 \textstyle\sum\limits_{l=0}^{m} \dbinom{m}{l} (\mu-n)^{m-l} 
 \Big\{D(1, -l; 1) - \textstyle\sum\limits_{k=2}^{n}\{\zeta(k)-1\}k^{l}\Big\}
 \\
= & 
 \textstyle\sum\limits_{l=0}^{m} \dbinom{m}{l} (\mu-n)^{m-l} 
 D(1, -l; 1)
   - \textstyle\sum\limits_{k=2}^{n}\{\zeta(k)-1\}\sum\limits_{l=0}^{m} 
   \dbinom{m}{l}(\mu-n)^{m-l} k^{l}
 \\
= & 
 \textstyle\sum\limits_{l=0}^{m} \dbinom{m}{l} (\mu-n)^{m-l} 
 D(1, -l; 1)
 - \textstyle\sum\limits_{k=2}^{n}\{\zeta(k)-1\}(k+\mu-n)^{m}.
\end{aligned}
\end{equation*}
This compleats the proof of Theorem~2.6.
\end{proof}

\section{Functional relations}

In this section, we prove Theorem~2.7.
In classical analytic number theory, functional equations for zeta functions
are often obtained by shifting the contour in a Hankel-type integral representation and evaluating the resulting residues.
A similar phenomenon arises in the present setting.
Starting from the integral representation (2.1) in Theorem~2.1,
we can derive a functional relation for $D(s_1,s_2;\lambda)$.

Alternatively, since $D(s_1,s_2;\lambda)$ admits an explicit expression in terms
of the Lerch transcendent, the same relation also follows directly from the
known transformation formula of the latter.

\begin{proof}[Proof of Theorem~2.7]
We recall the Lerch transcendent (cf.~\cite[Section~1.11]{erdelyi})
\begin{equation*}
\Psi(z,s,\lambda)
:=
\sum_{m=0}^{\infty}\frac{z^m}{(m+\lambda)^s},
\end{equation*}
which converges absolutely for $|z|<1$, or for $|z|=1$ with $\Re(s)>1$,
where we assume $\lambda \notin \mathbb{Z}_{\leqslant 0}$.
By Definition~\((1.1)\), we observe immediately that
\[
D(s_1,s_2;\lambda)
=
\sum_{n=2}^{\infty}
\frac{1}{n^{s_1+1}}
\Psi\!\left(\frac{1}{n},s_2,\lambda+1\right).
\]
Using the transformation formula for the Lerch transcendent
\[
\Psi(z,s,\lambda)
=
z^{-\lambda}\Gamma(1-s)
\sum_{l\in\mathbb{Z}}
(2\pi i l-\log z)^{\,s-1}
e^{2\pi i l\lambda},
\]
which holds for $\Re(s)<0$, $|z|<1$, and $0<\lambda \leqslant 1$
(see \cite[1.11 (6)]{erdelyi}),
we obtain the desired functional relation.
\end{proof}

%
%


\end{document}